\documentclass[11pt]{amsart} 

\usepackage[utf8]{inputenc} 
\usepackage{amsfonts}
\usepackage{amsopn}
\usepackage{hyperref}
\usepackage{amsthm}
\usepackage{amsmath}
\usepackage{tikz-cd}
\usepackage{amssymb}
\usepackage{anyfontsize}

\newtheorem{theorem}{Theorem}
\newtheorem{proposition}[theorem]{Proposition}

\newtheorem{lemma}[theorem]{Lemma}

\newtheorem{remark}[theorem]{Remark}
\newtheorem{definition}[theorem]{Definition}
\newtheorem{example}[theorem]{Example}
\numberwithin{theorem}{section}
\newcommand\pa{{\mathfrak a}}
\newcommand\pb{{\mathfrak b}}

\newcommand{\red}{\textcolor{black}}

\usepackage{geometry} 
\usepackage{graphicx} 

\title{\red{A new proof of} monomialisation from 3-folds to surfaces}
\author{Yueting Jiang}
\address{Universit\'e Paris Cit\'e and Sorbonne Universit\'e, UFR de Math\'ematiques, Institut de Math\'ematiques de Jussieu-Paris Rive Gauche, UMR7586, F-75013 Paris, France.}
\email{yueting.jiang@imj-prg.fr}

\begin{document}

\begin{abstract}
    In this paper, we give a new proof of the foundational result, due to S. Cutkosky, \red{on the existence of a monomialisation of a morphism from a 3-fold to a surface. Our proof brings to the fore the notion of log-Fitting ideals, and requires us to develop new methods related to Rank Theorems and log-Fitting ideals.
    }
\end{abstract}

\maketitle

\section{Introduction}
Monomialisation of morphisms is an important problem in resolution of singularities, which is a variation of classical resolution of singularities for morphisms. Starting in the time of Zariski, there has been a huge body of work on the problem of monomialisation. Variations include works on complex surfaces \cite{akbulut2012topology}, toroidalisation, factorisation and semistable reduction \cite{abramovich2013weak}, \cite{abramovich1997weak}, \cite{abramovich2002torification}, \cite{adiprasito2018semistable}, \cite{wlodarczyk2003toroidal}, on desingularisation \cite{abramovich2020relativedesingularizationprincipalizationideals}, and on monomialisation 
 \cite{cutkosky2002monomialization}, \cite{Cutkosky2013}, \cite{denef2013monomialization}, \cite{cutkosky2007toroidalization}, \cite{cutkosky2017local}, \cite{bbmono}. Our work concerns the latter version, where we are required to simplify the morphism via simple transformations in the source and target, namely smooth blowings-up.

In this paper, we provide a new proof of S. Cutkosky's foundational result of a monomialisation of a dominant morphism from a 3-fold to a surface \cite{cutkosky2002monomialization,Cutkosky2013}, by relying on the theory of log-Fitting ideals (recently used in this context in \cite{bbmono}) and on a result from S. Cutkosky and O. Kashcheyeva that a \emph{strongly-prepared morphism} can be reduced to a monomial morphism \cite{cutkosky2004monomialization}. We hope that our proof will help to clarify the Theorem, highlight key ingredients which can be independently studied, and help in the future development of the problem. In fact, the monomialisation result is also known to hold true for dominant morphisms between smooth projective 3-folds
\cite{cutkosky2007toroidalization}, and there are local results in arbitrary dimensions \cite{cutkosky1999local}, \cite{cutkosky2017local} and \cite{bbmono}. A global version in dimension higher than 3 remains unknown.

More precisely, our main result is \red{a new proof of}:
\begin{theorem}[{\cite[Theorem 1.3]{cutkosky2002monomialization}}]\label{mainthm}
    Let $\Phi:(X,E)\rightarrow (S,F)$ be a dominant morphism of pairs (Definition \ref{def:MorphismPairs}) between $\mathbb{K}$-varieties (Definition \ref{kv})
    with dim$X=3$ and dim$S=2$. Then $\Phi$ admits a monomialisation as \red{recalled}
    in Definition \ref{MOM}, i.e. there exists a commutative diagram:
    \begin{center}
	\begin{tikzcd}
		X \arrow[r, "\Phi"] \arrow[d, "\sigma"]
		& S \arrow[d, "\tau"] \\
		X^{\prime} \arrow[r, "\Phi^{\prime}"]
		& S^{\prime}
	\end{tikzcd}.
\end{center}
with $\sigma,\tau$ finite sequences of blowings-up and $\Phi^{\prime}$ a monomial morphism (Definition \ref{AMM}).
\end{theorem}

\red{The main novelty of this paper lies in the fact that our proof brings to the fore the notion of \emph{logarithmic Fitting ideals} following \cite{da2017resolution,bbmono}; their use had already been suggested for a connected problem in \cite{pardon2001pure} and previously studied by Teissier \cite{Teissier1977TheHO}. Log-Fitting ideals are a great tool to provide intrinsic characterisation of several notions linked to morphisms, including morphisms of Hsiang-Pati form, see e.g. \cite[Theorem 1.2]{bbmono}. More specifically, log-Fitting ideals naturally lead us to break the problem into three main parts.}

The first part consists in reducing an arbitrary dominant morphism of finite type to \emph{log-rank adapted} morphisms, a notion which we introduce in Definition \ref{wm}. This first step can be accomplished in arbitrary dimensions (see Theorem \ref{lem:LrAdapted}), and is the main technical novelty of this paper. Note that the notion of log-rank adapted morphisms provides a generalisation of the conditions obtained in \cite[Theorem 3.4]{da2017resolution} for birational morphisms. In fact, we followed the idea of the proof of \cite[Theorem 3.4]{da2017resolution}, but we need to add new arguments concerning the rank (and the log-rank) of morphisms. In fact, we need an algebraic geometry version of the classical differential geometry Rank Theorem, valid in the real and complex cases, for more general fields. We therefore develop tools related to Rank Theorems in algebraic geometry (see Proposition \ref{rankthm}), following the work of Rond \cite{rond2009homomorphisms}.

\red{The second part consists in reducing a log-rank adapted morphism from a $3$-fold to a surface to a \emph{strongly prepared morphism}, a notion introduced by Cutkosky in \cite{cutkosky2002monomialization}, and studied by Cutkosky and Kashcheyeva in \cite{cutkosky2004monomialization}. In fact, we start by providing an intrinsic characterisation of strongly prepared morphisms in terms of log-Fitting ideals: see Proposition \ref{light}. This allows us to show how the proof from \cite{da2017resolution} can be adapted, in a simple way, to prove this second reduction; see Theorem \ref{ESP}. This result makes the main result of \cite{da2017resolution} more flexible, allowing a possibly larger scope of applications. Cutkosky proves a similar reduction via other methods in \cite{cutkosky2002monomialization}.}

\red{Finally, we rely on Cutkosky and Kashcheyeva's proof that a strongly prepared morphism can be reduced to a monomial morphism \cite{cutkosky2004monomialization} as a black box to conclude the proof of the Theorem.}

\medskip

The paper is organised as follows. In Section \ref{preli} we recall some main definitions and the notion of locally principal ideal sheaves, while in Section \ref{grrt} we first introduce the notion of geometric rank and then establish a version of the Rank Theorem that will be used afterwards. Then we introduce quasi-prepared morphisms and strongly prepared morphisms in the two sections that follow. After that, we define the key concept of log-Fitting ideals in Section \ref{ssec:LogFittingCharacterizationStronglyPrepared} as well as the logarithmic rank. In Section \ref{lam}, we define log-rank adapted morphisms and make the reduction from quasi-prepared morphisms to log-rank adapted morphisms. Finally, the proof of the main theorem is briefly summarised in the last section.

\section{\red{Preliminaries}}
\label{preli}
\subsection{Review \red{of main} definitions}

\begin{definition}[$\mathbb{K}$-variety]
\label{kv}
    Let $\mathbb{K}$ be an algebraically closed field of characteristic zero. A $\mathbb{K}$-variety is a separated, reduced scheme of finite type over $\mathbb{K}$, not necessarily irreducible.
\end{definition}

\begin{remark}
    Blowing-up is projective for $\mathbb{K}$-varieties. Hence, a $\mathbb{K}$-variety remains a $\mathbb{K}$-variety after blowings-up. Moreover, a projective morphism of Noetherian schemes is proper, hence of finite type.
\end{remark}

Let us also recall several notions \red{from \cite{cutkosky2004monomialization,bbmono}}.

\begin{definition}[Morphism of pairs]\label{def:MorphismPairs}
    Let $\Phi:X\rightarrow Y$ be a morphism of finite type from a smooth $\mathbb{K}$-variety $X$ to a smooth $\mathbb{K}$-variety $Y$ with reduced simple normal crossing (SNC) divisors $E$ on $X$ and $F$ on $Y$ such that $\Phi^{-1}(F)_{\text{red}}\subseteq E$. Then $\Phi$ is called a morphism of pairs, denoted by $\Phi:(X,E)\rightarrow (Y,F)$. We will write the dimension of $X$ as $n$ and the dimension of $Y$ as $N$.
\end{definition}

\red{By} $E$-adapted regular (respectively \'{e}tale or formal) coordinates $\left(\vec{u},\vec{v}\right) = (u_1,\ldots,u_s,v_1,\\\ldots,v_{n-s})$ centred at a point $\pa\in X$, we mean local coordinates in $\mathcal{O}_{X,\pa}$ (respectively in $\mathcal{O}_{X,\pa}^h$ or $\widehat{\mathcal{O}}_{X,\pa}$) such that $E$ is locally equal to $\{u_1\cdots u_s=0\}$ with $s\in\mathbb{N}^*$. In this case, we say that $\pa$ is a $s$-point. Furthermore, a polynomial $f\in \mathbb{K}[u_1,\ldots,u_s,v_1,\ldots,v_{n-s}]$ will be called monomial with respect to $E$ if $f\in \mathbb{K}[u_1,\ldots,u_s]$, similarly for $f$ being an algebraic power series or a formal power series. In most cases, we will omit `with respect to $E$'.

For convenience, we will denote $F$-adapted coordinates in the target $S$ via the notation $\left(\vec{x},\vec{y}\right) =(x_1,\ldots,x_t,y_1,\ldots,y_{N-t})$ where $F = \{x_1 \cdots x_t=0\}$.

\begin{definition}[Dominant algebraic monomial morphism, see {\cite[Definition 1.1]{cutkosky2004monomialization}}]\label{AMM}
    Let $\Phi:X\rightarrow Y$ be a dominant morphism between smooth $\mathbb{K}$-varieties. A morphism $\Phi$ is called monomial if $\forall\pa\in X,\exists U$ an étale neighbourhood of $\pa$, $\left(w_1,\ldots,w_n\right)$ uniformising parameters on $U$, regular parameters $z_1,\dots,z_m\in \mathcal{O}_{Y,\Phi\left(\pa\right)}$ and a matrix $\left(a_{ij}\right)$ of non-negative integers with rank $m$ such that
$$
\left\{
\begin{aligned}
    z_1&=w_1^{a_{11}}\ldots w_n^{a_{1n}},\\
    \vdots\\
    z_m&=w_1^{a_{m1}}\ldots w_n^{a_{mn}}.
\end{aligned}
\right.
$$
Note that the $\left(w_1,\ldots,w_n\right)$ and $\left(z_1,\dots,z_m\right)$ are not necessarily centred at $\pa$ and $\Phi\left(\pa\right)$.

More generally, a morphism of pairs $\Phi: (X,E) \to (Y,F)$ is monomial if $\left(w_1,\ldots,w_n\right)$ on $U$ and $z_1,\dots,z_m\in \mathcal{O}_{Y,\Phi\left(\pa\right)}$ are $E$-adapted and $F$-adapted respectively.
\end{definition}

\begin{remark}[On the definition of monomial morphisms, cf. {\cite[Remark 1.2]{bbmono}}]
\label{defbbmono}\hfill
\begin{enumerate}
\item In \cite[Defintion 1.1 and 1.3]{bbmono}, monomial morphisms are defined for non-dominant morphisms. But we only treat dominant morphisms in this \red{paper}, so we prefer to work with Definition \ref{AMM}.
    \item Similarly, the definition of monomial morphisms from \cite[Defintion 1.1 and 1.3]{bbmono} includes the analytic category, which requires one to include translations into the normal form. See details about the relationship between the definitions in \cite[Remark 1.2]{bbmono}.
    \item In the algebraic case, we have the same form as above for $\Phi$. But now $\left(u,v\right)$ should be thought of as uniformising parameters on an étale neighbourhood and $\left(x,y\right)$ as regular parameters exactly as in Definition \ref{AMM}.
\end{enumerate}
\end{remark}

\begin{definition}[Monomialization of morphisms]\label{MOM}
    Let $\Phi:X\rightarrow Y$ be a dominant morphism between smooth $\mathbb{K}$-varieties. Then $\Phi$ admits a monomialisation if there exists a commutative diagram:
    \begin{center}
	\begin{tikzcd}
		X^{\prime} \arrow[r, "\Phi^{\prime}"] \arrow[d, "\sigma"]
		& Y^{\prime} \arrow[d, "\tau"] \\
		X \arrow[r, "\Phi"]
		& Y
	\end{tikzcd}
\end{center}
with $\sigma,\tau$ finite sequences of blowings-up and $\Phi^{\prime}$ a monomial morphism. Moreover, if $X,Y$ are equipped with divisors, we ask that $\sigma,\tau$ adapted to the divisors and $\Phi^{\prime}$ is a morphism of pairs.
\end{definition}

\subsection{Locally principal ideal sheaves}
In this subsection, we consider the notion of locally principal ideal sheaf in our context. It turned out that we only need to require that the ideal sheaf is locally principal on all the closed points.

More precisely, we have the following lemma. It could be understood as a consequence of the bijection of constructible sets between the maximal spectrum and prime spectrum of a finitely generated algebra over a field. In other words, it is due to the correspondence of constructible sets between the classic algebraic variety and the scheme both obtained from a finitely generated algebra over a field.

\begin{lemma}\label{inssup}
    Let $X$ be a scheme of finite type over a field and $\mathcal{I}$ an ideal sheaf of $\mathcal{O}_X$. Then the $\mathcal{O}_X$-module $\mathcal{I}$ is invertible if and only if it is locally free of rank 1 at any closed point of $X$. If this is the case, the ideal sheaf $\mathcal{I}$ is called locally principal.
\end{lemma}
\begin{proof}
Since $X$ is a locally ringed space, the $\mathcal{O}_X$-module $\mathcal{I}$ is invertible if and only if it is locally free of rank 1. Hence, the direct implication is obvious.

Now we prove the inverse implication. Suppose that $\mathcal{I}$ is locally free of rank 1 at any closed point of $X$. Let $\eta\in X$ be a non-closed point; $\overline{\{\eta\}}$ is an irreducible closed subset of $X$. Take $\overline{\{\eta\}}$ as a closed subscheme of $X$, of which the scheme structure does not matter. It has a closed point since every non-empty quasi-compact scheme has a closed point $x$. The point $x$ is also closed in $X$. Therefore, $\exists U\subset X$ an open neighbourhood of $x$ such that $\mathcal{I}$ is of rank 1 on $U$. On the other hand, $\eta\in U\cap\overline{\{\eta\}}$ as $U\cap\overline{\{\eta\}}$ is an open neighbourhood of $x$ in $\overline{\{\eta\}}$. So $\exists U\subset X$ an open neighbourhood of $\eta$ such that $\mathcal{I}$ is of rank 1.
\end{proof}

For completeness, we present \cite[\href{https://stacks.math.columbia.edu/tag/01WR}{Definition 01WR}]{stacks-project} in the following definition.

\begin{definition}
Let $S$ be a scheme.

A locally principal closed subscheme of $S$ is a closed subscheme whose sheaf of ideals is locally generated by a single element.

An effective Cartier divisor on $S$ is a closed subscheme $D\subset S$ whose ideal sheaf $\mathcal{I}_D\subset \mathcal{O}_S$ is an invertible $\mathcal{O}_S$-module.
\end{definition}

\section{\red{The geometric rank and the Rank Theorem}}
\label{grrt}
We recall several necessary definitions, cf. \cite[1.1. \textbf{Terminology.}]{rond2009homomorphisms}.
\begin{definition}[Homomorphism of local rings]
A homomorphism of local rings $\varphi: A \longrightarrow B$ is a ring homomorphism such that $\varphi\left(\mathfrak{m}_A\right) \subset \mathfrak{m}_B$ and the extension of fields $A / \mathfrak{m}_A \longrightarrow B / \mathfrak{m}_B$ is finite.
\end{definition}

\begin{remark}
	The above definition adds the condition that the extension of residue fields is finite to the usual definition of homomorphism of local rings. This is because that we work with henselian local rings, and henselian local rings are always about finite field extensions.
\end{remark}

\begin{definition}[Local $\mathbb{K}$-algebra]
    A local $\mathbb{K}$-algebra is a finitely generated $\mathbb{K}$-algebra $A$ with a Noetherian local ring structure such that the homomorphism $\mathbb{K} \longrightarrow A$ is injective and the extension of fields $\mathbb{K} \longrightarrow A / \mathfrak{m}_A$ is finite.
\end{definition}

When we say homomorphism of local $\mathbb{K}$-algebras, we refer to homomorphism of the underlying local ring structures of the local $\mathbb{K}$-algebras concerned.

Recall that, roughly speaking, the rank of a homomorphism between classic algebraic varieties at a given point is the rank of the Jacobian matrix evaluated at this point over $\mathbb{K}$. Here we give another kind of rank, the so-called geometric rank. The geometric rank is a nice invariant that provides information on the structure of the homomorphism.

\begin{definition}[Geometric rank of local $\mathbb{K}$-algebras]\label{GRK}
    Let $A,B$ be local $\mathbb{K}$-algebras. The geometric rank of $\varphi: A \longrightarrow B$ is equal to the rank of the $B$-module generated by $\Omega_A^1$ in $\Omega_B^1$. We shall note the geometric rank of $\varphi$ as grk$\left(\varphi\right)$.
\end{definition}

\begin{remark}
Definition \ref{GRK} coincides with the definition given by Jacobian matrix in the classic setting of algebraic varieties.
\end{remark}

Now we could give a slightly more general definition of the rank by using the definition of geometric rank.

\begin{definition}[Rank of local $\mathbb{K}$-algebras]\label{RK}
    Let $A,B$ be local $\mathbb{K}$-algebras. The rank of $\varphi: A \longrightarrow B$ is equal to the rank of the $B / \mathfrak{m}_B$-module generated by $\Omega_A^1\otimes A / \mathfrak{m}_A$ in $\Omega_B^1\otimes B / \mathfrak{m}_B$. We shall note the geometric rank of $\varphi$ as rk$\left(\varphi\right)$.
\end{definition}

Consider $\mathbb{K}$-varieties $X,Y$, and let $\Phi:X\rightarrow Y$ be a morphism. Let $x\in X$ be a closed point; its image $\Phi(x)$ is also a closed point, because $X,Y$ are both schemes of finite type over $\mathbb{K}$. By referring to classical results of localisation and henselisation, one sees that $\mathcal{O}^h_{X,x}$ and $\mathcal{O}^h_{Y,\Phi(x)}$ are both local $\mathbb{K}$-algebras. If $X,Y$ are both smooth, then local rings $\mathcal{O}^h_{X,x}$ and $\mathcal{O}^h_{Y,\Phi(x)}$ are regular. In this case, Defintion \ref{GRK} still coincides with the definition given by Jacobian matrix, so it is locally constant at closed points.

\begin{remark}
    The geometric rank is called generic rank in \cite{belotto2021proof}.
    It can be generalised to positive characteristics as well: cf. \cite[Definition 2.1]{rond2009homomorphisms}.
\end{remark}

Next we give an important lemma on the geometric rank.

\begin{lemma}\label{ins}
    Let $\Phi:X\rightarrow Y$ be a morphism between $\mathbb{K}$-varieties and $E$ an irreducible component of $X$. Then the geometric rank is constant on $E$.
\end{lemma}
\textit{Proof.} As remarked above, for $X,Y$ smooth $\mathbb{K}$-varieties, the definition of geometric rank is locally constant at closed points. Hence, $\forall x\in E$ closed, we could find an open neighbourhood $U_x\subset E$ of $x$ such that the geometric rank is constant on $U_x$. Note that $E$ consists of closed points and a unique generic point $\eta$ with $\eta\in U_x$. Let $x^{\prime}\neq x\in E$ be another closed point; we have $U_x\cap U_{x^{\prime}}\neq\emptyset$ since $E$ is irreducible, so the geometric rank is constant over $\cup_{x\in E\text{ closed}}U_x=E$.
\qed

\begin{remark}
    Note that, $\forall i$, the geometric rank of the unique generic point $\eta$ of $E_i$ is defined in the course of the proof of Lemma \ref{ins}.
    
    Inspired by the proof of Lemma \ref{ins} and the proof of Lemma \ref{inssup}, we expand the definition of geometric rank to non-closed points of a smooth $\mathbb{K}$-variety $X$ as follows. Let $\eta\in X$ be a non-closed point. $\overline{\{\eta\}}$ is an irreducible closed subset of $X$. Take $\overline{\{\eta\}}$ as a closed subscheme of $X$, of which the scheme structure does not matter. It has a closed point $x$, since every non-empty quasi-compact scheme has a closed point. The geometric rank at $\eta$ is defined as the geometric rank at $x$. One can check that it is well-defined, i.e. $\forall x^{\prime}\neq x\in X$ a closed point and $\forall U_{x^{\prime}}\subset X$ an open neighbourhood of $x^{\prime}$ containing $\eta$ over which the geometric rank is constant, the geometric rank of $\eta$ defined by $x^{\prime}$ is the same as that defined by $x$.
\end{remark}

Recall Chevalley's Theorem, cf. \cite[Corollaire (I.8.5)]{grothendieck1964elements}:
\begin{theorem}[\red{Chevalley}]
    Let $\mathrm{Y}$ be a Noetherian scheme and $f: \mathrm{X} \rightarrow \mathrm{Y}$ a morphism of finite type. For all constructible subsets $\mathrm{Z}$ of $\mathrm{X}, f(\mathrm{Z})$ is a constructible subset of $\mathrm{Y}$.
\end{theorem}

Recall that $X,Y$ are both schemes of finite type over $\mathbb{K}$, so the image of a closed point is a closed point. On the other hand, closed points of $X,Y$ are of the form $V\left(z_1-\lambda_1,\ldots,z_n-\lambda_n\right)$ for $\lambda_1,\ldots,\lambda_n\in\mathbb{K}$. Consequently, the sets of closed points of $X,Y$ are non-empty.

\begin{proposition}\label{rankthm}
    Let $\mathbb{K}$ be an algebraically closed field and $f: X \longrightarrow Y$ a morphism of finite type between $\mathbb{K}$-varieties. Note $Z=\overline{f(X)}$, the Zariski closure of $f(X)$. If max$\{\text{grk}_x(f)\}_{\{x\in X\text{ closed}\}}=r\in\mathbb{N}$, then dim $f(X)=$dim$Z=r$.
\end{proposition}
\begin{proof} By Chevalley's Theorem, $f(X)$ is constructible. Since $Y$ is a Noetherian scheme, the underlying topological space $|Y|$ is Noetherian. Then $f(X)=\cup_{i\in I} F_i \cap U_i$, where $I$ is finite; note that $\left(F_i\right)_{i\in I}$ are closed subsets of $Y$ and $\left(U_i\right)_{i\in I}$ are open subsets of $Y$. By decomposing each $F_i$ as a Noetherian topological space to a union of irreducible closed subsets, we can suppose that each $F_i$ is irreducible. Every $F_i \cap U_i$ is non-empty, hence dense in $F_i$. So $Z=\cup_{i\in I} F_i$. There is a unique subset $J\subseteq I$ such that each $F_j$ is an irreducible component of $Z$ and $F_j,j\in J$ are not pairwise comparable under inclusion.

If $k$ is a field and $X$ is a $k$-scheme of finite type, all subschemes of $X$ are of finite type over $k$. Thus, the closed subscheme $Z$ is a scheme of finite type over $\mathbb{K}$. Consequently, $Z$ is Noetherian, and the dimension of $Z$ is the maximum of the dimensions of its irreducible components. Note that $f(X)$ is a non-empty dense open subset in $Z$, so dim$Z=$dim$f(X)$. Therefore, it suffices to prove that dim$Z=r$.

Write $Z=\cup_{j\in J} F_j$. Dimension is defined as a topological property, which is not influenced by the algebraic structure in question. Hence, $\forall j\in J$, put the reduced closed subscheme structure on $F_j$; then $F_j$ is an integral scheme of finite type over $\mathbb{K}$, and thus dim$F_j$=dim$\mathcal{O}_{F_j,P}$ for any closed point $P$ of $F_j$. However, closed points are dense in $F_j$, so $F_j\cap U_j$ contains a closed point $y\in Y$. Consider $f^{-1}(y)$, a non-empty closed subset of $X$: it has a closed point $x^{\prime}$ since every non-empty quasi-compact scheme has a closed point. If grk$_{x^{\prime}}(f)=r^{\prime}$, $\exists U^{\prime}$ an open neighbourhood of $y$ in $Y$ such that $Z$ is of dimension $r^{\prime}$ in $U^{\prime}$ by {[Theorem 6.7]\cite{rond2009homomorphisms}}. It follows that $F_j$ is of dimension less than or equal to $r^{\prime}$ in $U^{\prime}$, and dim$\mathcal{O}_{F_j\cap U^{\prime},y}\leq r^{\prime}$. Therefore, dim$\mathcal{O}_{F_j,y}$=dim$\mathcal{O}_{F_j\cap U^{\prime},y}\leq r^{\prime}$. Therefore, we have dim$Z\leq r$.

Let $x$ be a closed point in $X$ such that grk$_x(f)=r$, $f(x)$ is closed in $Y$ and $f(x)\in F_{i_0}\cap U_{i_0}\subseteq Z$ for some $i_0\in I$. $\exists U$ an open neighbourhood of $f(x)$ in $Y$ such that $Z$ is of dimension $r$ in $U$ by {[Theorem 6.7]\cite{rond2009homomorphisms}}. Note that $\cup_{z\in Z}U_z\cup U$ is an open covering of $Z$. Closed points are dense in $Z$, so every $U_z$ contains a closed point. By a similar argument as that in the last paragraph, dim$Z\leq r$ in every $U_z$. Therefore, dim$Z=r$.
\end{proof}

\section{\red{Quasi-prepared morphisms}}
\begin{definition}[Quasi-prepared morphism, see {\cite[Definition 3.1]{cutkosky2004monomialization}}]\label{def:QuasiPrepared}
    Consider a morphism of pairs (not necessarilly dominant) $\Phi:(X,E)\rightarrow (Y,F)$ and let
    \[
    \mbox{Sing} \left(\Phi\right):= \{\pa\in X ; \, \mbox{rk}_{\pa}(\Phi)< \mbox{dim}Y =  N\}
    \]
    be the locus of singular points of $\Phi$. We will say that $\Phi$ is quasi-prepared if Sing$\left(\Phi\right)\subseteq E$ and $\Phi^{-1}(F)_{\text{red}}=E$. 
\end{definition}

\begin{lemma}[Reduction to quasi-prepared morphism]\label{lem:BasicQuasiPrepared}
Let $\Phi: X \to Y$ be a dominant morphism between $\mathbb{K}$-varieties with $n\geq N$ (as usual, $n,N$ denote the dimensions of $X$ and $Y$ respectively) and reduced SNC divisors $E$ on $X$, $F$ on $Y$, such that $\overline{\Phi\left(E\cup\text{Sing}(X)\right)}$ has codimension $\geq 1$.
Then there exist finite sequences of blowings-up $\sigma: (X',E')\to (X,E), \tau: (Y',F') \to (Y,F)$, and a dominant quasi-prepared morphism $\Phi':(X',E') \to (Y',F')$ such that the following diagram commutes:
\begin{center}
	\begin{tikzcd}
		X^{\prime} \arrow[r, "\Phi^{\prime}"] \arrow[d, "\sigma"]
		& Y^{\prime} \arrow[d, "\tau"] \\
		X \arrow[r, "\Phi"]
		& Y
	\end{tikzcd}
\end{center}
\end{lemma}
\begin{proof} By resolution of singularities, we can assume that $Y$ is smooth. Again by the resolution of singularities and its functoriality, we can assume that $X$ is smooth. The functoriality of the resolution of singularities ensures the existence of $\Phi^{\prime}$.

By further blowing up $\mbox{Sing}\left(\Phi\right)$, we can assume that $\mbox{Sing}\left(\Phi\right)\subseteq E$.

Let us consider $Z$, the Zariski closure of $\Phi(E)$. It is a proper subscheme of $Y$ by hypothesis. Thus, we may suppose $\Phi(E)\subseteq F$ by resolution of singularities of $Z$ over $Y$. Moreover we have $\left(X,E\right)\rightarrow\left(Y,F\right)$ with $\Phi^{-1}(F)=E$ by resolving $\mathcal{I}_F$, the ideal sheaf of $F$.

All relevant modifications are blowings-up, so $\Phi^{\prime}$ is still dominant. Blowing-up is projective for $\mathbb{K}$-varieties. Hence, a $\mathbb{K}$-variety remains a $\mathbb{K}$-variety after blowings-up. Hence, $X,Y$ are still $\mathbb{K}$-varieties after all the modifications in this proof. Furthermore, $\Phi^{\prime}$ is of finite type because $\sigma\circ\Phi$ is locally of finite type and $X^{\prime}$ is a Noetherian scheme.
\end{proof}

\begin{lemma}[Stability of quasi-prepared morphisms by blowings-up]\label{lem:PreserveQuasi-prepared}
    Let $\Phi:(X,E)\rightarrow (Y,F)$ be a quasi-prepared morphism of pairs. Suppose that $\sigma:(X',E') \to (X,E)$ is a blowing-up with centre contained in $E$. Then $\Phi'= \Phi \circ \sigma$ is quasi-prepared.
\end{lemma}
\begin{proof}
We have $\left(\Phi^{\prime}\right)^{-1}(F)=E^{\prime}$ because all the centres of blowing-up are contained in $E$. On the other hand, a blowing-up is an
isomorphism outside its centre, so the rank does not change outside of
it. Therefore, a quasi-prepared morphism of pairs is still quasi-prepared after blowings-up with centres contained in $E$.
\end{proof}

\section{\red{Strongly prepared morphisms}}
As a key preliminary step \red{towards monomialisation}, Cutkosky and Kashcheyeva consider the notion of strongly prepared:

\begin{definition}[Strongly prepared morphism onto a surface, see {\cite[Definition 3.2]{cutkosky2004monomialization}}]\label{spm}
Suppose that $\Phi:(X,E)\to(S,F)$ is a dominant quasi-prepared morphism from a smooth $n$-fold $X$ to a smooth surface $S$. We will say that $\Phi$ is strongly prepared at $\pa\in X$ if there exist $E$-adapted formal coordinates $(u, v)$ at $a$ and $F$-adapted coordinates
$(x_1,\ldots,x_n)$ in $\mathcal{O}_{Y,\Phi\left(\pa\right)}$ such that one of the following forms holds:
\begin{enumerate}
    \item $1\leq k\leq n-1$: $\pa$ is a $k$-point,
    \[
    x_1=\left(\vec{u}^{\vec{\alpha}}\right)^m,\quad z=P\left(\vec{u}^{\vec{\alpha}}\right)+\vec{u}^{\vec{\beta}}v_1,
    \]
    where $z$ is either divisorial $(z=x_2)$ or free $(z=y_1)$, $m>0$, $\alpha_i>0,\forall i\in\{1,\ldots,k\}$ with $\gcd\left(\alpha_1,\ldots,\alpha_k\right)=1$, $\beta_1,\ldots,\beta_k\geq 0$ and $P$ is a formal series;
    \item $2\leq k\leq n-1$: $a$ is a $k$-point,
    \[
    x_1=\left(\vec{u}^{\vec{\alpha}}\right)^m,\quad z=P\left(\vec{u}^{\vec{\alpha}}\right)+\vec{u}^{\vec{\beta}},
    \]
    where $z$ is either divisorial $(z=x_2)$ or free $(z=y_1)$, $m>0$, $\alpha_i>0,\forall i\in\{1,\ldots,k\}$ with $\gcd\left(\alpha_1,\ldots,\alpha_k\right)=1$, $\beta_1,\ldots,\beta_k\geq 0$, $\vec{\alpha}\wedge\vec{\beta}\neq 0$ and $P$ is a formal series;
    \item $2\leq k\leq n-1$: $a$ is a $k$-point,
    \[
    x_1=u_1^{\alpha_1}\cdots u_{k-1}^{\alpha_{k-1}},\quad x_2=u_2^{\beta_2}\cdots u_k^{\beta_k},
    \]
    where $\alpha_2,\ldots,\alpha_{k-1},\beta_2,\ldots,\beta_{k-1}\geq 0$, $\alpha_1,\beta_k>0$ and $\alpha_i+\beta_i>0,\forall i\in\{2,\ldots,k-1\}$.
\end{enumerate}
\end{definition}

\begin{example}
	For examples of strongly prepared morphism onto a surface, we have:
    \noindent
	Case (1),
	$$
	\left\{
	\begin{aligned}
		x_1&=\left(u_1u_2\right)^2\\
		y_1&=u_1u_2+u_1^3u_2^3v_1.
	\end{aligned}
	\right.
	$$
	    \noindent
	Case (2),
	$$
	\left\{
	\begin{aligned}
		x_1&=\left(u_1u_2\right)^2\\
		y_1&=u_1u_2+u_1^2u_2^3.
	\end{aligned}
	\right.
	$$
        Case (3),
        $$
	\left\{
	\begin{aligned}
		x_1&=u_1u_2^2\\
		x_1&=u_2^3u_3^4.
	\end{aligned}
	\right.
	$$
\end{example}

In fact, they have proved the following result.
\begin{theorem}[{\cite[Theorem 18.19]{cutkosky2002monomialization}}]\label{blackbox}
Suppose that $\Phi:(X,E)\rightarrow (S,F)$ is a strongly prepared morphism from a smooth $3$-fold $X$ to a smooth surface $S$. Then $\Phi$ admits a monomialisation.
\end{theorem}

\red{In Proposition \ref{light} below, we will provide an \emph{intrinsic} characterization of strongly prepared morphisms.}

\section{Log-Fitting ideals and log-rank}\label{ssec:LogFittingCharacterizationStronglyPrepared}
\subsection{\red{Definition of log-Fitting ideals}}

Let $\left(X,E\right)$ be a smooth $\mathbb{K}$-variety (of dimension $n$) along with a reduced SNC divisor. We denote by $\Omega_X^1$ the $\mathcal{O}_X$-module of 1-forms on $X$ and by $\Omega_X^1\left(\text{log}E\right)$ the $\mathcal{O}_X$-module of logarithmic 1-forms on $X$. Let $\pa\in X$ and let $\left(\vec{u},\vec{v}\right)$ be $E$-adapted \'etale or analytic coordinate system at $\pa$. The sheaf $\Omega_X^1\left(\text{log}E\right)$ is locally free at $\pa$ with basis given by
$$
\frac{du_i}{u_i},i=1,\ldots,s,\text{ and }dv_j,j=1,\ldots,n-s.
$$
The logarithmic Jacobian matrix is defined analogously to Jacobian matrix using these bases. The sheaf $\Omega^k_X\left(\text{log}E\right)$ of logarithmic $k$-forms over $X$ is defined in terms of $\Omega^1_X\left(\text{log}E\right)$ in the standard way.

Let $\Phi:(X,E)\rightarrow (Y,F)$ be a morphism of pairs. For all $k=1, \ldots, N$, $\Phi^*\left(\Omega_Y^k(\log F)\right)$ is a subsheaf of $\Omega_X^k(\log E)$. For the proof, see {\cite[Lemma 4.1]{bbmono}}. One has a slogan: \textbf{Pullbacks of $\log$ differentials are well-defined}. For the definition of logarithmic Fitting ideal sheaf $\mathcal{F}_{n-k}(\Phi)$ associated to $\Phi$, see {\cite[Definition 4.2]{bbmono}}. We restate it here using our notation.
\begin{definition}\label{LFD}
For any $k\in\{1,\ldots,\mbox{min}\{n,N\}\}$, the logarithmic Fitting ideal sheaf $\mathcal{F}_{n-k}(\Phi)$ associated to $\Phi$ is the ideal subsheaf of $\mathcal{O}_X$ whose stalk $\mathcal{F}_{n-k}(\Phi)_{\pa}$ at $\pa \in X$ can be described (in $E$-adapted coordinates at $\pa$) in the following way: if $\omega\in\left(\Omega_Y^k\left(\text{log}F\right)\right)_{\Phi\left(\pa\right)}$, then
$$
\Phi^* \omega=\sum_{I, J} B_{I, J}^\omega(\boldsymbol{u}, \boldsymbol{v}) \frac{d u_{i_1}}{u_{i_1}} \wedge \cdots \wedge \frac{d u_{i_l}}{u_{i_l}} \wedge d v_{j_1} \wedge \cdots \wedge d v_{j_{k-l}}
$$
where the sum is over all pairs $(I, J)$ with $I=\left(i_1, \ldots i_l\right), 1 \leq i_1<\cdots<i_l \leq s$, $J=\left(j_1, \ldots<j_{k-l}\right), 1 \leq j_1<\cdots j_{k-l} \leq n-s$ and $B_{I, J}^\omega\in\mathcal{O}_{X,\pa}$, and $\mathcal{F}_{n-k}(\Phi)_{\pa}$ is generated by the set $\{B_{I, J}^\omega(\boldsymbol{u}, \boldsymbol{v})\}_{I,J,\omega\in\Omega_Y^k\left(\text{log}F\right)}$.
\end{definition}

\begin{remark}
    Note that, for a morphism $\Phi$, being quasi-prepared implies immediately $V\left(\mathcal{F}_{0}\left(\Phi\right)\right)\subseteq E$. Consequently, $V\left(\mathcal{F}_{k}\left(\Phi\right)\right)\subseteq E$ since $V\left(\mathcal{F}_{k}\left(\Phi\right)\right)\subseteq V\left(\mathcal{F}_{0}\left(\Phi\right)\right)$, $\forall 0<k<N$.
\end{remark}

\subsection{Definition of log-rank}
\begin{definition}
    Let $\Phi:(X,E)\rightarrow (Y,F)$ be a morphism of pairs. $\forall\pa\in X$, the logarithmic rank of $\Phi$ at $\pa$ is defined as the rank of the logarithmic Jacobian matrix logrk$_{\pa}\left(\Phi\right)$.
\end{definition}

\begin{lemma}\label{Coincident}
    Let $\Phi:\left(X,E\right)\rightarrow\left(Y,\emptyset\right)$ be a morphism between smooth $\mathbb{K}$-varieties and $\pa$ a closed $s$-point with $\pa\in D:=E_1\cap\ldots\cap E_s$. Then logrk$_{\pa}\left(\Phi\right)=\mbox{grk}_{\pa}\Phi|_D$.
\end{lemma}
\begin{proof}
It suffices to check log-Jacobian matrix and Jacobian matrix corresponding to logrk$_{\pa}\left(\Phi\right)$ and $\mbox{grk}_{\pa}\Phi|_D$ respectively.
\end{proof}

\subsection{\red{Intrinsic characterisation of strongly prepared morphisms}}
\red{We are now ready to provide an intrinsic characterisation of strongly prepared morphisms. We start with the following well-known result:}

\begin{lemma}\label{formalpmtopm}
\red{Let $X$ be a scheme, $\pa \in X$ and $\mathcal{I}\subset\mathcal{O}_{X,\pa}$ be an ideal sheaf. Then $\mathcal{I}$ is principal monomial if and only if $\mathcal{I}\cdot\widehat{\mathcal{O}}_{X,\pa}$ is principal monomial.}
\end{lemma}
\begin{proof}	
\red{Note that the forward implication is direct. So, let us assume that $\mathcal{I}\cdot\widehat{\mathcal{O}}_{X,\pa} $ is principal monomial and prove the other implication.}

Notice that $\mathcal{O}_{X,\pa}\rightarrow \widehat{\mathcal{O}_{X,\pa}}$ is flat because the completion of a Noetherian local ring with respect to its maximal ideal is faithfully flat; see \cite[\href{https://stacks.math.columbia.edu/tag/00MC}{Lemma 00MC}]{stacks-project}. Therefore $\mathcal{I}\cdot\widehat{\mathcal{O}_{X,\pa}}=\mathcal{I}\otimes_{\mathcal{O}_{X,\pa}}\widehat{\mathcal{O}_{X,\pa}}$. Furthermore, $\widehat{\mathcal{O}_{X,\pa}}/\mathfrak{m}_{\pa}\widehat{\mathcal{O}_{X,\pa}}=\mathcal{O}_{X,\pa}/\mathfrak{m}_{\pa}$ by construction. Thus, we have
	$$
	\begin{aligned}
		\mathcal{I}\cdot\widehat{\mathcal{O}_{X,\pa}}\otimes_{\widehat{\mathcal{O}_{X,\pa}}}\widehat{\mathcal{O}_{X,\pa}}/\mathfrak{m}_{\pa}\widehat{\mathcal{O}_{X,\pa}}&=\left(\mathcal{I}\otimes_{\mathcal{O}_{X,\pa}}\widehat{\mathcal{O}_{X,\pa}}\right)\otimes_{\widehat{\mathcal{O}_{X,\pa}}}\mathcal{O}_{X,\pa}/\mathfrak{m}_{\pa}\\&=\mathcal{I}\otimes_{\mathcal{O}_{X,\pa}}\mathcal{O}_{X,\pa}/\mathfrak{m}_{\pa}.
	\end{aligned}
	$$
	Note that $\mathcal{I}\cdot\widehat{\mathcal{O}_{X,\pa}}\otimes_{\widehat{\mathcal{O}_{X,\pa}}}\widehat{\mathcal{O}_{X,\pa}}/\mathfrak{m}_{\pa}\widehat{\mathcal{O}_{X,\pa}}$ is principal by hypothesis. Consequently, $\mathcal{I}\otimes_{\mathcal{O}_{X,\pa}}\mathcal{O}_{X,\pa}/\mathfrak{m}_{\pa}=\mathcal{I}/\mathcal{I}\mathfrak{m}_{\pa}$ is principal. This means that $\exists f\in\mathcal{I}\setminus \mathcal{I}\mathfrak{m}_{\pa}$ such that $\mathcal{I}/\mathcal{I}\mathfrak{m}_{\pa}=\bar{f}\mathcal{O}_{X,\pa}$, where $\bar{f}$ is $f$ modulo $\mathcal{I}\mathfrak{m}_{\pa}$. Then $\mathcal{I}=f\mathcal{O}_{X,\pa}+\mathcal{I}\mathfrak{m}_{\pa}$ and $\mathcal{I}/f\mathcal{O}_{X,\pa}=\mathfrak{m}_{\pa}\mathcal{I}/f\mathcal{O}_{X,\pa}$. By Nakayama's Lemma, $\mathcal{I}/f\mathcal{O}_{X,\pa}=0$, and thus $\mathcal{I}=f\mathcal{O}_{X,\pa}$. If $\mathcal{I}\cdot\widehat{\mathcal{O}}_{X,\pa}$ is monomial, $f$ is also monomial.
\end{proof}

\red{We may now prove the desired characterisation, following the methods from Theorem 4.4 of \cite{bbmono} and Lemma 3.1 from \cite{da2017resolution}.}

\begin{proposition}\label{light}
Let $\Phi: (X,E) \to (S,F)$ be a quasi-prepared morphism onto a surface. Then, $\Phi$ is strongly prepared if and only if $\mathcal{F}_{0}(\Phi)$ is locally principal monomial.
\end{proposition}
\begin{proof}
We use $E,F$-adapted coordinates.

If $\Phi$ is strongly prepared, one easily check that $\frac{dx_1}{x_1}\wedge \frac{dx_2}{x_2}$ or $\frac{dx_1}{x_1}\wedge dy_2$ is formally principal monomial by definition of strongly prepared morphism onto a surface. Then by Lemma \ref{formalpmtopm}, they are principal monomial. Besides, $V\left(\mathcal{F}_{0}\left(\sigma\right)\right)\subseteq E$ since the generator of $\mathcal{F}_{0}$ is monomial on $\{u_1,\ldots,u_s\}$.

We shall prove the other implication in what follows.

Note $\pb=\phi(\pa)$. There are two cases: either $\pb\in S$ is a 1-point, or $\pb\in S$ is a 2-point.

\begin{enumerate}
\item\label{case12} If $\pb$ is a 1-point, let $U$ an open neighbourhood of $\pb$ such that $\phi\left(E\cap U\right)\subseteq F_1$ with $F_1$ an irreducible component of $F$. Since $\phi^{-1}(F)_{red}=E$ by the definition of quasi-prepared morphism, $\Phi^{-1}\left(F_1\right)\cap U=E\cap U$. Suppose that $x_1=0$ is the local equation of $F$ containing $\pb$; then $x_1$ can be factored out of all the exponents of local coordinates, and the rest must be a unit because the base field is algebraically closed and $\phi^{-1}\left(F_1\right)\cap U=E\cap U$. By making an étale change of coordinates, one can write $x_1=\left(u_1^{\alpha_1}\cdots u_k^{\alpha_k}\right)^m$ with $\left(\alpha_1,\ldots,\alpha_k\right)=1$ and $\alpha_1,\ldots,\alpha_k,m\in \mathbb{N}^*$.

Then we can write
$$
\left\{
\begin{aligned}
    x_1&=\left(u_1^{\alpha_1}\ldots u_k^{\alpha_k}\right)^m,m\in\mathbb{N}^*\\
    y_1&=P\left(u_1^{\alpha_1}\ldots u_k^{\alpha_k}\right)+R\left(u_1,\ldots,u_k,v_1,\ldots,v_{n-k}\right).
\end{aligned}
\right.
$$
Now we make a formal expansion and write
$$
R\left(u_1,\ldots,u_k,v_1,\ldots,v_{n-k}\right)=\sum_{\vec{\gamma}\in\mathbb{N}^n}R_{\vec{\gamma}}u_1^{\gamma_1}\ldots u_k^{\gamma_k}v_1^{\gamma_k+1}\ldots v_{n-k}^{\gamma_n},
$$
where $\vec{\gamma}\wedge \vec{\alpha}=0$ (note $\vec{\gamma}=\left(\gamma_1,\ldots,\gamma_n\right)$) and $R_{\vec{\gamma}}\in\mathbb{K}^*$. Now
$$
\begin{aligned}
    \frac{dx_1}{x_1}\wedge dy_1&=\frac{d\left(u_1^{\alpha_1}\ldots u_k^{\alpha_k}\right)^m}{\left(u_1^{\alpha_1}\ldots u_k^{\alpha_k}\right)^m}\wedge d\sum_{\vec{\gamma}\in\mathbb{N}^n}R_{\vec{\gamma}}u_1^{\gamma_1}\ldots u_k^{\gamma_k}v_1^{\gamma_k+1}\ldots v_{n-k}^{\gamma_n}\\
    &=\sum_{\vec{\gamma}\in\mathbb{N}^n}\sum_{i,j\in\{1,\ldots,k\}}mR_{\vec{\gamma}}u_1^{\gamma_1}\ldots u_k^{\gamma_k}v_1^{\gamma_k+1}\ldots v_{n-k}^{\gamma_n}\left(\alpha_i\gamma_j-\alpha_j\gamma_i\right)\frac{du_i}{u_i}\wedge\frac{du_j}{u_j}\\
    &+\sum_{\vec{\gamma}\in\mathbb{N}^n}\sum_{\substack{i\in\{1,\ldots,k\}\\l\in\{1,\ldots,n-k\}}}mR_{\vec{\gamma}}\frac{u_1^{\gamma_1}\ldots u_k^{\gamma_k}v_1^{\gamma_k+1}\ldots v_{n-k}^{\gamma_n}}{v_l}\alpha_i\gamma_l\frac{du_i}{u_i}\wedge dv_l
\end{aligned}
$$
If $\mathcal{F}_0\left(\phi\right)$ is principal monomial, $\mathcal{F}_0\left(\phi\right)=\vec{u}^{\vec{\gamma_0}}$. Then $u_1^{\gamma^0_1}\ldots u_k^{\gamma^0_k}v_1^{\gamma^0_{k+1}}\ldots v_{n-k}^{\gamma^0_n}$ divides $\frac{dx_1}{x_1}\wedge dy_1$. Hence $\vec{\gamma^0}|\vec{\gamma},\forall\vec{\gamma}$.

If $\vec{\gamma_0}$ comes from the first double summation of $\frac{dx_1}{x_1}\wedge dy_1$, then $\gamma^0_l=0,l\in\{k+1,\ldots,n\}$, since $V\left(\mathcal{F}_{0}\left(\sigma\right)\right)\subseteq E$. We have $R=\sum_{\vec{\gamma}\in\mathbb{N}^n}R_{\vec{\gamma}}u_1^{\gamma^0_1}\ldots u_k^{\gamma^0_k}v_1^{\gamma^0_{k+1}}\ldots v_{n-k}^{\gamma^0_n}\theta$ with $\theta$ a unit. We are in the case (1) of Definition \ref{spm} by making an étale change of coordinates.

When $\vec{\gamma_0}$ can not come from the first summation, we have that $\forall \vec{\gamma}$, $\exists l\in\{1,\ldots,n-k\}$ such that $\vec{\gamma^0}|\frac{\vec{\gamma}}{v_l}$. Furthermore, $\exists \vec{\gamma^1},l_1$ such that $\vec{\gamma^0}=\frac{\vec{\gamma_1}}{x_{l_1}}$ with $l_1\in\{1,\ldots,n-k\}$ and $\gamma^1_l=0$, $\forall l\neq l_1\in\{1,\ldots,n-k\}$ (several such $\vec{\gamma^1},l_1$ might exist). Write $R=\vec{x}^{\vec{\gamma^0}}f\left(u_1,\ldots,u_k,v_1,\ldots,v_{n-k}\right)$. Therefore, $\exists l\in\{1,\ldots,n-k\}$ such that $\partial_{x_l}f(0)\neq 0$ (one could choose $l_1$ for example). Then we can use the implicit function theorem and it turns out that we are in the case (2) of Definition \ref{spm}.

\item If $\pb$ is a 2-point, then $x_1x_2=0$ is the local equation of $F$ at $\pb$, and the pullback of $x_1x_2$ is monomial on $\left(u_1,\ldots,u_k,v_1,\ldots,v_{n-k}\right)$. We have
$$
\left\{
\begin{aligned}
    x_1&=u_1^{\alpha_1}\ldots u_k^{\alpha_k}\zeta\\
    x_2&=u_1^{\beta_1}\ldots u_k^{\beta_k}\eta,
\end{aligned}
\right.
$$
where $\zeta$ and $\eta$ are units, $\alpha_2,\ldots,\alpha_{k-1},\beta_2,\ldots,\beta_{k-1}\geq 0$, $\alpha_1,\beta_k>0$ and $\alpha_i+\beta_i>0,\forall i\in\{2,\ldots,k-1\}$. If $\vec{\alpha}\wedge\vec{\beta}\neq 0$, necessarily $k\geq 2$. By eventually alternating $u_1,\ldots,u_k$, we can assume that $\left(\alpha_1,\alpha_2\right)$ and $\left(\beta_1,\beta_2\right)$ are linearly independent. Write
$$
\left\{
\begin{aligned}
    x_1&=\tilde{x}_1v_1\\
    x_2&=\tilde{x}_2v_2,
\end{aligned}
\right.
$$
where $v_1,v_2$ are indeterminates that help to get rid of $\zeta$ and $\eta$.

After substituting $x_i,i=1,2$ by $\tilde{x}_iv_i$, we see that it suffices to solve
\[ \begin{bmatrix}
    \alpha_1 & \alpha_2\\
    \beta_1 & \beta_2
\end{bmatrix}
\begin{bmatrix}
    v_1\\
    v_2
\end{bmatrix}=
\begin{bmatrix}
    \zeta^{-1}\\
    \eta^{-1}
\end{bmatrix} \]
The solution is evident by taking the inverse of the matrix
\[ \begin{bmatrix}
    \alpha_1 & \alpha_2\\
    \beta_1 & \beta_2
\end{bmatrix} \]
on both sides of the above equation. We are in case (3) of Definition \ref{spm}.

When $\vec{\alpha}\wedge\vec{\beta}=0$, since $\vec{\alpha}\neq\vec{0}$, we could absorb $\zeta$ into some $u_i$ with $\alpha_i\neq 0$ and write
$$
\left\{
\begin{aligned}
	x_1&=\left(u_1^{\alpha_1}\ldots u_k^{\alpha_k}\right)^m,m\in\mathbb{N}^*\\
	x_2&=P\left(u_1^{\alpha_1}\ldots u_k^{\alpha_k}\right)+R\left(u_1,\ldots,u_k,v_1,\ldots,v_{n-k}\right).
\end{aligned}
\right.
$$
Then we make a similar analysis as that in \ref{case12} and we get cases (1) and (2) of Definition \ref{spm}.
\end{enumerate}
This finishes the proof.
\end{proof}

\section{\red{Log-rank adapted morphisms}}
\label{lam}
We define log-rank adapted morphism following the discussion with A. Belotto. This kind of morphism is essential to our proof, and turns out to be useful in many other contexts as well.

Let $\Phi:(X,E)\rightarrow (Y,F)$ be a morphism of pairs. For any $i\in 1,\ldots,\mbox{min}\{n,N\}$, note $\Sigma_i=\{\pa\in E:logrk_{\pa}\left(\Phi\right)\leq \mbox{min}\{n,N\}-i\}$ and $Y_i=\overline{\Phi\left(\Sigma_i\right)}$. Observe that $\Sigma_i$ is a $\mathbb{K}$-variety by definition. The ideal sheaf of $Y_i$ is denoted by $\mathcal{I}_{Y_i}$.

\begin{definition}[Log-rank adapted morphism]\label{wm}
	Let $\Phi:(X,E)\rightarrow (Y,F)$ be a quasi-prepared morphism (not necessarilly dominant), where $\mbox{dim}X=n$ and $\mbox{dim}Y=N$. \red{Denote by $\Phi_0:(X,E) \rightarrow (Y,\emptyset)$ the morphism where the divisor $F$ is declared to be empty.} 
    
    We say that $\Phi$ is log-rank adapted at a point $\pa\in E$ if we can choose $E$-adapted coordinates $\left(\vec{u},\vec{v}\right)$ at $\pa$ and local coordinates $z_1,\ldots,z_N$ at $\Phi\left(\pa\right)$ such that
    \begin{enumerate}
        \item $\Phi=\left(\Phi_1,\ldots,\Phi_N\right)$ satisfies
	$$
	\Phi_1=v_1,\ldots,\Phi_r=v_r,\Phi_{r+1}=\vec{u}^{\vec{\alpha_1}},
	$$
	where $r$ is the log-rank of \red{$\Phi_0$} at $\pa$, $\vec{\alpha_1}\in\mathbb{N}^s$, $\mathcal{I}_{Y_{N-r}}$ is generated by $z_{r+1},\ldots,z_N$ at $\Phi\left(\pa\right)$ and \red{$\Phi_0^*\mathcal{I}_{Y_{N-r}}$} is generated by $\Phi_{r+1}=\Phi^*\left(z_{r+1}\right)$ at $\pa$.
        \item There is a finite sequence of SNC divisors $E^{(k)}$ such that $E^{(1)}=E$, $E^{(k+1)}\subseteq E^{(k)}$, $\forall k=1,\ldots,\mbox{min}\{n,N\}$ and
        \begin{enumerate}
            \item logrk$_{\pa}\left(\Phi\right)=\mbox{min}\{n,N\}-k$, $\forall \pa\in E^{(k)}\setminus E^{(k+1)}$;
            \item $\forall\pa\in E^{(k)}\setminus E^{(k+1)}$, $\Phi\left(E^{(k)}\right)$ is smooth at $\Phi\left(\pa\right)$.
        \end{enumerate}
    \end{enumerate}
\end{definition}

We prove that quasi-prepared morphisms can be made log-rank adapted.

\begin{theorem}[Reduction to log-rank adapted morphism]\label{lem:LrAdapted}
    Let $\Phi: (X,E) \to (Y,F)$ be a \red{quasi-prepared morphism, where $\mbox{dim}X=n$ and $\mbox{dim}Y=N$.} 
\red{Then there exists a sequence of blowings-up $\sigma: (X',E') \to (X,E)$, whose centres are contained in $E$ and its pull-backs, such that the composed morphism $\Phi' = \Phi \circ \sigma$ is everywhere log-rank adapted.}
\end{theorem}
\begin{proof}
At closed points in $E$, the log-rank coincides with the geometric rank restricted to $E$ by Lemma \ref{Coincident}. Thus, by Proposition \ref{rankthm}, dim$Y_1\leq \mbox{min}\{n,N\}-1$.

After resolving $\mathcal{I}_{Y_1}$, we again define $\Sigma_i=\{\pa\in E:logrk_{\pa}\left(\Phi\right)\leq \mbox{min}\{n,N\}-i\}$ and $Y_i=\overline{\Phi\left(\Sigma_i\right)}$ for $i=1,2$. Then we resolve $Y_2^{\prime}=Y_2\cup\mbox{Sing}Y_1$ and the morphism after resolution is still denoted by $\Phi$. It turns out that $\Sigma_1=\Phi^{-1}\left(Y_1\right)$. Hence, we define $E^{(1)}$ as $\Phi^{-1}\left(Y_1\right)$. Note that $E^{(1)}=E$ by definition of log-rank. The inverse image $\Phi^{-1}\left(Y_2^{\prime}\right)$ is an SNC divisor. By definition of log-rank, logrk$_{\pa}\left(\Phi\right)\leq\mbox{dim}Y_2^{\prime}\leq\mbox{min}\{n,N\}-2$, $\forall\pa\in\Phi^{-1}\left(Y_2^{\prime}\right)$. Define $E^{(2)}=\Phi^{-1}\left(Y_2^{\prime}\right)$. We get the second condition of Definition \ref{wm} for $k=1$.

For $k=2$, we define $\Sigma_3$ and $Y_3$ analogously and resolve $Y_3\cup\mbox{Sing}Y_2$. Then we define $E^{(1)},E^{(2)},E^{(3)}$ in a similar way.

Continue in the way of last paragraph, we achieve the second condition of Definition \ref{wm} for $k=1,\ldots,\mbox{min}\{n,N\}$.

Finally, by the constancy of log-rank over $E^{(k)}\setminus E^{(k+1)}$, smoothness of $\Phi\left(E^{(k)}\right)$ at $\Phi\left(\pa\right)$ and prinicpality of $\Phi_0^*\mathcal{I}_{Y_{N-r}}$, one can make étale change of coordinates as in \cite[Theorem 3.7]{da2017resolution} to get the first condition of Definition \ref{wm}.
\end{proof}

\red{Finally, the definition of log-rank adapted morphism is exactly what is needed to extend the main result of \cite{da2017resolution}, i.e. Theorem 1.3 of \cite{da2017resolution}. More precisely,}

\begin{theorem}\label{ESP}
\red{Let $\Phi: (X,E) \to (S,F)$ be a log-rank adapted morphism where $X$ is a $3$-fold and $S$ is a surface. There exists a finite sequence of blowings-up $\sigma: (X',E') \to (X,E)$, whose centres are over $E$
, such that the composed morphism $\Phi' = \Phi \circ \sigma$ is log-rank adapted and $\mathcal{F}_1(\Phi')$ is principal monomial. In particular, $\Phi' = \Phi \circ \sigma$ is everywhere strongly prepared.}
\end{theorem}
\begin{proof}
\red{By the definition of log-rank adapted, the ideal $\mathcal{F}_2(\Phi)$ is everywhere principal monomial, and $\mathcal{F}_1(\Phi)$ is principal monomial at every point $\pa \in X$ where the log-rank equals to $1$. We must, therefore, principalise $\mathcal{F}_1(\Phi)$ over points $\pa$ where the log-rank is $0$. We may now follow the proof of Theorem 1.1 from \cite{da2017resolution}, but with the following difference: the log-Fitting ideal $\mathcal{F}_0(\Phi)$ does not exist. This difference only intervenes in the proof of Lemma 6.4 in \cite{da2017resolution}. However, it does not matter because, firstly, the ideal sheaf $\mathcal{H}_a$ mentioned in the proof of Lemma 6.4 is well-defined in our context (see Lemma \ref{HA}); secondly, the 2$\times$2 minors, which are incarnations of 3$\times$3 minors mentioned in the proof of Lemma 6.4, are still supported in $E$.}
\end{proof}

\begin{lemma}\label{HA}
    The ideal sheaf $\mathcal{H}_a$ mentioned in the proof of \cite[Lemma 6.4]{da2017resolution} is well-defined in our context.
\end{lemma}
\begin{proof}
    Use the notations in \cite{da2017resolution}. It suffices to prove that $\mathcal{H}_a\neq 0$ at a $1$-point $a\in X$. Now we only have $\sigma_1$ and $\sigma_2$. So it suffices to prove that $a_{20}\left(u,w\right)$ in the form (4.2) of \cite{da2017resolution} is non-zero, we will do it by contradiction.

    Assume that $a_{20}\left(u,w\right)=0$. Then $d\sigma_1\wedge d\sigma_2=du^{\alpha}\wedge du^{\delta}T_2=u^{\alpha+\delta}\left(\frac{T_2}{v}+v\partial_v\left(\frac{T_2}{v}\right)\right)\frac{du^{\alpha}}{u^{\alpha}}\wedge dv$. By the form (4.2), we write:
\[
\frac{T_2}{v}=\tilde{T}_2\left(u,v,w\right)v^{d-1}+\sum_{j=1}^{d-1}a_{2j}\left(u,w\right)v^{j-1}.
\]
    By hypothesis, $d>1$. So there is always a non-zero solution of $\frac{T_2}{v}+v\partial_v\left(\frac{T_2}{v}\right)=0$ outside the exceptional divisor locally defined by $u=0$, which contradicts the fact that $V\left(\mathcal{F}_1\right)\subseteq E$.
\end{proof}

\section{Proof of \red{Theorem \ref{mainthm}}}
In this section, we prove the main theorem \ref{mainthm}.

\begin{proof}
We take a morphism of pairs, reduce it to a log-rank adapted morphism by Lemma \ref{lem:BasicQuasiPrepared} and Theorem \ref{lem:LrAdapted}, then make it strongly prepared by Theorem \ref{ESP}. In light of Lemma \ref{light}, we use Theorem \ref{blackbox} to finish the proof.
\end{proof}

\noindent
\textbf{Acknowledgements.} During the preparation of this work, the author was supported by the project ``Plan d’investissements France 2030", IDEX UP ANR-18-IDEX-0001 at Université Paris Cité. The associated laboratory is Institut de Mathématiques de Jussieu-Paris Rive Gauche. I gratefully thank my advisor André Belotto for his careful review and for his many insightful comments that helped me to
improve this paper. I am also grateful to Guillaume Rond for his several very helpful comments and discussions on the algebro-geometric part of the paper, as well as Edward Bierstone for his contributive comments.

\bibliographystyle{amsalpha}
\bibliography{references}

\providecommand{\bysame}{\leavevmode\hbox to3em{\hrulefill}\thinspace}
\providecommand{\MR}{\relax\ifhmode\unskip\space\fi MR }
\providecommand{\MRhref}[2]{%
  \href{http://www.ams.org/mathscinet-getitem?mr=#1}{#2}
}
\providecommand{\href}[2]{#2}
\begin{thebibliography}{BdSBGM17}

\bibitem[ADK13]{abramovich2013weak}
Dan Abramovich, Jan Denef, and Kalle Karu, \emph{Weak toroidalization over
  non-closed fields}, Manuscripta mathematica \textbf{142} (2013), 257--271.

\bibitem[AK00]{abramovich1997weak}
Dan Abramovich and Andrew~K Karu, \emph{Weak semistable reduction in
  characteristic 0}, vol. 139, Invent. math., 2000, pp.~241–--273.

\bibitem[AK12]{akbulut2012topology}
Selman Akbulut and Henry King, \emph{Topology of real algebraic sets}, vol.~25,
  Springer Science \& Business Media, 2012.

\bibitem[AKMW02]{abramovich2002torification}
Dan Abramovich, Kalle Karu, Kenji Matsuki, and Jaros{\l}aw W{\l}odarczyk,
  \emph{Torification and factorization of birational maps}, Journal of the
  American Mathematical Society \textbf{15} (2002), no.~3, 531--572.

\bibitem[ALT18]{adiprasito2018semistable}
Karim Adiprasito, Gaku Liu, and Michael Temkin, \emph{Semistable reduction in
  characteristic 0}, arXiv preprint arXiv:1810.03131 (2018).

\bibitem[ATW20]{abramovich2020relativedesingularizationprincipalizationideals}
Dan Abramovich, Michael Temkin, and Jarosław Włodarczyk, \emph{Relative
  desingularization and principalization of ideals}, arXiv preprint
  arXiv:2003.03659 (2020).

\bibitem[BdSBGM17]{da2017resolution}
Andr{\'e} Belotto~da Silva, Edward Bierstone, Vincent Grandjean, and Pierre~D
  Milman, \emph{Resolution of singularities of the cotangent sheaf of a
  singular variety}, Advances in Mathematics \textbf{307} (2017), 780--832.

\bibitem[BdSCR21]{belotto2021proof}
Andr{\'e} Belotto~da Silva, Octave Curmi, and Guillaume Rond, \emph{A proof of
  a. gabrielov’s rank theorem}, Journal de l’{\'E}cole
  polytechnique—Math{\'e}matiques \textbf{8} (2021), 1329--1396.

\bibitem[CK04]{cutkosky2004monomialization}
Steven~Dale Cutkosky and Olga Kashcheyeva, \emph{Monomialization of strongly
  prepared morphisms from nonsingular n-folds to surfaces}, Journal of Algebra
  \textbf{275} (2004), no.~1, 275--320.

\bibitem[Cut99]{cutkosky1999local}
Steven~Dale Cutkosky, \emph{Local monomialization and factorization of
  morphisms}, Astérisque \textbf{260} (1999).

\bibitem[Cut02]{cutkosky2002monomialization}
\bysame, \emph{Monomialization of morphisms from 3-folds to surfaces}, vol.
  1786, Springer Science \& Business Media, 2002.

\bibitem[Cut07]{cutkosky2007toroidalization}
\bysame, \emph{Toroidalization of dominant morphisms of 3-folds}, American
  Mathematical Soc., 2007.

\bibitem[Cut13]{Cutkosky2013}
\bysame, \emph{A simpler proof of toroidalization of morphisms from 3-folds to
  surfaces}, Annales de l’institut Fourier \textbf{63} (2013), no.~3,
  865--922 (eng).

\bibitem[Cut17]{cutkosky2017local}
\bysame, \emph{Local monomialization of analytic maps}, Advances in Mathematics
  \textbf{307} (2017), 833--902.

\bibitem[Den13]{denef2013monomialization}
Jan Denef, \emph{Monomialization of morphisms and p-adic quantifier
  elimination}, Proceedings of the American Mathematical Society \textbf{141}
  (2013), no.~8, 2569--2574.

\bibitem[dSB23]{bbmono}
Andr{\'e}~Belotto da~Silva and Edward Bierstone, \emph{{Monomialization of a
  quasianalytic morphism}}, {Annales Scientifiques de l'{\'E}cole Normale
  Sup{\'e}rieure} (2023), 66 pages; revised version, theorems unchanged; to
  appear in Ann. Sci. Ecole Norm. Sup.

\bibitem[Gro64]{grothendieck1964elements}
Alexander Grothendieck, \emph{{\'E}l{\'e}ments de g{\'e}om{\'e}trie
  alg{\'e}brique: Iv. {\'e}tude locale des sch{\'e}mas et des morphismes de
  sch{\'e}mas, premi{\`e}re partie}, Publications Math{\'e}matiques de
  l'IH{\'E}S \textbf{20} (1964), 5--259.

\bibitem[PS01]{pardon2001pure}
William Pardon and Mark Stern, \emph{Pure hodge structure on the l2-cohomology
  of varieties with isolated singularities}, Journal für die reine und
  angewandte Mathematik (2001).

\bibitem[Ron09]{rond2009homomorphisms}
Guillaume Rond, \emph{Homomorphisms of local algebras in positive
  characteristic}, Journal of Algebra \textbf{322} (2009), no.~12, 4382--4407.

\bibitem[{Sta}24]{stacks-project}
The {Stacks project authors}, \emph{The stacks project},
  \url{https://stacks.math.columbia.edu}, 2024.

\bibitem[Tei77]{Teissier1977TheHO}
Bernard Teissier, \emph{The hunting of invariants in the geometry of
  discriminants: Five lectures at the nordic summer school}, 1977.

\bibitem[W{\l}o03]{wlodarczyk2003toroidal}
Jaros{\l}aw W{\l}odarczyk, \emph{Toroidal varieties and the weak factorization
  theorem}, Inventiones mathematicae \textbf{154} (2003), no.~2, 223--331.

\end{thebibliography}

------------------------------------------------------------------------------

\end{document}